\documentclass[11pt]{article}
\usepackage{amsthm, amsmath, amssymb, amsfonts, url, booktabs, tikz, setspace, fancyhdr, bm}
\usepackage{geometry}
\usepackage{hyperref, enumerate}
\usepackage[shortlabels]{enumitem}
\usepackage[babel]{microtype}
\usepackage[english]{babel}
\usepackage[capitalise]{cleveref}
\usepackage{comment}
\usepackage{bbm}
\usepackage{csquotes}
\usepackage{mathabx}
\usepackage{tikz}
\usepackage{graphicx}
\usepackage{float}
\usepackage[dvipsnames]{xcolor}
\usepackage{soul}
\usepackage{mathtools}
\usetikzlibrary{positioning, arrows.meta, shapes.geometric}
\usepackage[normalem]{ulem}
\counterwithin{figure}{section}

\newtheorem{theorem}{Theorem}[section]
\newtheorem{prop}[theorem]{Proposition}
\newtheorem{conj}[theorem]{Conjecture}
\newtheorem{lemma}[theorem]{Lemma}
\newtheorem{cor}[theorem]{Corollary}

\usetikzlibrary{decorations.pathmorphing}
\theoremstyle{definition}

\newtheorem*{defn-non}{Definition}

\newlist{Case}{enumerate}{2}
\setlist[Case, 1]{%
label = {\bfseries Case \arabic*.},
labelindent=1em ,labelwidth=1.3cm, labelsep*=1em, leftmargin =!
}
\setlist[Case, 2]{%
label = {\bfseries Subcase \arabic{Casei}.\arabic*.},
labelindent=-1em ,labelwidth=1.3cm, labelsep*=1em, leftmargin =!
}

\usepackage{todonotes}

\title{Layer barriers for colour-biased tight Hamilton cycles}
\author{
Zijian Deng\thanks{School of Mathematics and Statistics, Lingnan Normal University, Zhanjiang, China. E-mail: zjl329205716@163.com}.
\and
Qinfei Tang\thanks{Department of Mathematics Research, Fujian Institute of Education, Fuzhou, China. E-mail: tqf9500@126.com}.
\and
Caihong Yang\thanks{School of Mathematics and Physics, China University of Geosciences, Wuhan, China. Email: yangch@cug.edu.cn.}
}
\date{}

\begin{document}
\maketitle

\begin{abstract}
We construct a family of layer barriers for colour-biased tight Hamilton
cycles in uniform hypergraphs. For every $k\ge 3$ and every
$a\in\{0,\ldots,k-1\}$, we give a red--blue coloured $k$-graph that
contains a tight Hamilton cycle, while every tight Hamilton cycle in the
construction is perfectly colour-balanced. The construction underlying
the higher-uniformity threshold conjectured by Behague, Clemen, Hyde and
Morrison corresponds to the boundary case $a=0$ of this family. We show
that interior choices of $a$ can yield strictly denser barriers. In
particular, for $k=17$ and $a=8$, the asymptotic relative minimum vertex
degree of our construction is
\[
    \frac{5761}{8192}\approx 0.703247,
\]
which exceeds the conjectured value
$d_{17}\approx 0.699277$. This provides a counterexample to the proposed
higher-uniformity threshold in Conjecture~6.1 of Behague, Clemen, Hyde
and Morrison. Moreover, by choosing the layer appropriately as
$k\to\infty$, the family contains barriers whose asymptotic relative
minimum vertex degree is
\[
    1-O\bigl(k^{-1/2}\bigr).
\]
Thus the interior members of the layer-barrier family exhibit
substantially different behaviour from the previously considered
boundary construction in large uniformity.
\end{abstract}


\section{Introduction}\label{sec:intro}

Discrepancy questions for spanning structures ask for quantitative strengthenings of classical existence theorems: once a host graph or hypergraph is sufficiently dense to contain a prescribed spanning object, must every edge-colouring contain such an object whose colour distribution is substantially non-uniform? This viewpoint has led to discrepancy versions of a number of extremal and Hamiltonicity results. In graphs, recent work has treated spanning trees and Hamilton cycles \cite{balogh2020discrepancies,gishboliner2022spanning}, colour-biased Hamilton cycles in dense and random graphs \cite{freschi2021colour,gishboliner2022random}, powers of Hamilton cycles \cite{bradac2022powers}, and general $H$-factors \cite{bradac2024factors}. These results show that the threshold forcing a spanning structure and the threshold forcing a linearly colour-biased copy of that structure may coincide in some settings and differ in others. In a complementary direction, Chen, Rong and Xu~\cite{CRX2025} established an optimal stability result for colour-biased Hamilton cycles in edge-coloured graphs. They showed that, under a minimum-degree condition close to the threshold for Hamiltonicity, if every Hamilton cycle has small colour-bias, then the coloured graph must be structurally close to the known extremal constructions. Their result reveals that structural stability can already emerge well below the extremal threshold for forcing colour imbalance.

The analogous programme for hypergraphs has developed rapidly. Colour-biased perfect matchings have been studied under several minimum-degree regimes \cite{balogh2024note,han2025colour}, while Gishboliner, Glock and Sgueglia \cite{gishboliner2025tight} determined the asymptotically optimal minimum codegree threshold for tight Hamilton cycles with high discrepancy in $r$-edge-coloured $k$-graphs for every $k\ge3$ and $r\ge2$. Minimum vertex degree is substantially more delicate. Even in the uncoloured setting, sharp or asymptotically sharp vertex-degree conditions for tight Hamilton cycles require significant structural work; see Reiher, R\"odl, Ruci\'nski, Schacht and Szemer\'edi \cite{reiher2019minimum} for $3$-graphs and Lang and Sanhueza-Matamala \cite{lang2022minimum} for a broader minimum-degree framework.

We use the following notation. A $k$-uniform hypergraph, or $k$-graph, is a hypergraph in which every edge has size $k$. For a $k$-graph $H$ and $1\le \ell\le k-1$, the minimum $\ell$-degree of $H$ is
\[
\delta_\ell(H):=\min_{\substack{S\subseteq V(H)\\ |S|=\ell}}
\bigl|\{e\in E(H):S\subseteq e\}\bigr|.
\]
In particular, $\delta_1(H)$ is the minimum vertex degree. A tight Hamilton cycle in an $n$-vertex $k$-graph is a cyclic ordering $v_1,\ldots,v_n$ of its vertices such that $\{v_i,v_{i+1},\ldots,v_{i+k-1}\}\in E(H)$ for every $i$, where subscripts are taken modulo $n$. Suppose that the edges of $H$ are coloured red and blue. Encode the colouring by
\[
c(e):=
\begin{cases}
1,&\text{if $e$ is red},\\
-1,&\text{if $e$ is blue},
\end{cases}
\qquad
c(F):=\sum_{e\in E(F)}c(e)
\]
for every subgraph $F\subseteq H$. Since a tight Hamilton cycle $C$ has exactly $n$ edges, it contains at least $(1/2+\delta)n$ edges of one colour if and only if $|c(C)|\ge 2\delta n$.

Behague, Clemen, Hyde and Morrison \cite{behague2026minimum} recently determined the asymptotically optimal minimum vertex-degree threshold for a red-blue coloured $3$-graph to contain a colour-biased tight Hamilton cycle. They then proposed a threshold for every higher uniformity. For $k\ge3$, set
\begin{equation}\label{eq:dk}
 d_k:=1-\frac{k-1}{2k}\left(1-\frac{1}{2k}\right)^{k-2}.
\end{equation}
Their Conjecture~6.1 extends to every $k\ge4$ a conjecture posed for $k=4$ in the perfect-matching work of H\`an et al.\ \cite{han2025colour}. We record it below as Conjecture~\ref{conj:BCHM}.

\begin{conj}\cite{behague2026minimum} \label{conj:BCHM}
For every $k\ge4$ and every $\alpha>0$, there exist $\delta>0$ and $n_0\in\mathbb N$ such that the following holds. If $H$ is a red-blue coloured $k$-graph on $n\ge n_0$ vertices and
\[
\delta_1(H)\ge (d_k+\alpha)\binom{n}{k-1},
\]
then $H$ contains a tight Hamilton cycle with at least $(1/2+\delta)n$ edges of one colour.
\end{conj}

The obstruction motivating Conjecture~\ref{conj:BCHM} partitions the vertex set into a large part and a part of size about $n/(2k)$, deletes the layer of $k$-sets meeting the smaller part in exactly two vertices, and colours the two adjacent layers so that every tight Hamilton cycle is colour-balanced. Its asymptotic relative minimum vertex degree is $d_k$. Our main observation is that this construction is only the boundary member of a natural one-parameter family. Moving the active layers into the interior allows a second forbidden layer to be used, and the resulting barrier can have a strictly larger minimum vertex degree.

To state the family, fix $k\ge3$ and $a\in\{0,\ldots,k-1\}$, and let
\[
X_{k,a}\sim\operatorname{Bin}\!\left(k-1,\frac{2a+1}{2k}\right).
\]
With $\mathbb P(X_{k,a}=j)=0$ understood for $j\notin\{0,\ldots,k-1\}$, define
\begin{equation}\label{eq:dka}
\begin{split}
 d_{k,a}:=1-\max\bigl\{&
 \mathbb P\bigl(X_{k,a}\in\{a-1,a+2\}\bigr),\\
 &\mathbb P\bigl(X_{k,a}\in\{a-2,a+1\}\bigr)
 \bigr\}.
\end{split}
\end{equation}
For $a=0$, the first probability in \eqref{eq:dka} is $\mathbb P(X_{k,0}=2)$ and the second is $\mathbb P(X_{k,0}=1)$. Since $X_{k,0}\sim\operatorname{Bin}(k-1,1/(2k))$,
\[
\frac{\mathbb P(X_{k,0}=2)}{\mathbb P(X_{k,0}=1)}
=\frac{k-2}{2(2k-1)}<1.
\]
Hence
\[
d_{k,0}=1-\mathbb P(X_{k,0}=1)=d_k.
\]
Thus \eqref{eq:dka} genuinely extends \eqref{eq:dk}.

\begin{theorem}\label{thm:main}
Let $k\ge3$, let $a\in\{0,\ldots,k-1\}$, and let $n$ be divisible by $2k$. There exists a red-blue coloured $k$-graph $H_{k,a}(n)$ on $n$ vertices with the following properties:
\begin{enumerate}[(i)]
\item $H_{k,a}(n)$ contains a tight Hamilton cycle;
\item every tight Hamilton cycle $C$ in $H_{k,a}(n)$ satisfies $c(C)=0$;
\item for fixed $k$ and $a$, as $n\to\infty$ through multiples of $2k$,
\[
\delta_1\bigl(H_{k,a}(n)\bigr)
=\bigl(d_{k,a}+o(1)\bigr)\binom{n}{k-1}.
\]
\end{enumerate}
\end{theorem}
At uniformity $17$, an interior member of the family already disproves Conjecture~\ref{conj:BCHM}.

\begin{cor}\label{cor:k17}
Conjecture~\ref{conj:BCHM} is false. More precisely,
\[
d_{17,8}=\frac{5761}{8192}
>1-\frac{8}{17}\left(\frac{33}{34}\right)^{15}=d_{17}.
\]
Consequently, there exists a fixed $\alpha>0$ and arbitrarily large red-blue coloured $17$-graphs $H$ satisfying
\[
\delta_1(H)\ge(d_{17}+\alpha)\binom{n}{16}
\]
such that every tight Hamilton cycle in $H$ has exactly $n/2$ red edges and $n/2$ blue edges.
\end{cor}
Interior layers also change the large-uniformity behaviour.
\begin{cor}\label{cor:largek}
As $k\to\infty$,
\[
1-\max_{0\le a\le k-1}d_{k,a}=O(k^{-1/2}),
\]
whereas
\[
d_k\longrightarrow 1-\frac{e^{-1/2}}{2}.
\]
\end{cor}

The rest of the paper is organised as follows. In Section~\ref{sec:barriers} we introduce the layer barriers, prove the layer-confinement property, and show both that every tight Hamilton cycle is perfectly colour-balanced and that the construction is Hamiltonian. In Section~\ref{sec:degree} we compute the minimum vertex degree, complete the proof of Theorem~\ref{thm:main}, and then prove Corollaries~\ref{cor:k17} and~\ref{cor:largek}.

\section{Layer barriers}\label{sec:barriers}

Fix $k\ge3$ and $a\in\{0,\ldots,k-1\}$, and suppose that $2k$ divides $n$. Partition a set $V$ of $n$ vertices as
\[
V=A\sqcup B,
\qquad
|B|=\frac{2a+1}{2k}n.
\]
The \emph{type} of a $k$-set $e\subseteq V$ with respect to the ordered partition $(A,B)$ is $|e\cap B|$. Define $H_{k,a}(n)$ by
\begin{equation}\label{eq:construction}
E\bigl(H_{k,a}(n)\bigr)
:=\left\{e\in\binom{V}{k}:|e\cap B|\notin\{a-1,a+2\}\right\}.
\end{equation}
Types outside $\{0,\ldots,k\}$ are void. Colour every edge of type $a$ red and every edge of type $a+1$ blue; colour all remaining edges arbitrarily. The two omitted types separate the possible intersection sizes into three intervals, and this simple separation is the mechanism behind the construction.
\begin{lemma}\label{lem:confinement}
Let $C$ be a tight cycle in $H_{k,a}(n)$, and write its edges in cyclic order as $e_1,\ldots,e_m$. Then
\[
\{|e_i\cap B|:1\le i\le m\}
\]
is contained in one of
\[
\{0,\ldots,a-2\},
\qquad
\{a,a+1\},
\qquad
\{a+3,\ldots,k\},
\]
where empty intervals are ignored.
\end{lemma}

\begin{proof}
Set $s_i:=|e_i\cap B|$, with indices taken modulo $m$. Consecutive edges of a tight cycle are consecutive windows of length $k$ in its cyclic vertex ordering. Thus one vertex leaves the window and one enters, so $|s_{i+1}-s_i|\le1$ for every $i$. By \eqref{eq:construction}, neither $a-1$ nor $a+2$ can occur among the $s_i$. Hence the cyclic sequence $s_1,\ldots,s_m$ cannot move between distinct connected components of
\[
\{0,\ldots,k\}\setminus\{a-1,a+2\}.
\]
These components are precisely the three intervals in the statement.
\end{proof}

\begin{prop}\label{prop:balanced}
Every tight Hamilton cycle $C$ in $H_{k,a}(n)$ has exactly $n/2$ red edges and $n/2$ blue edges.
\end{prop}

\begin{proof}
Let $e_1,\ldots,e_n$ be the edges of $C$ in cyclic order and set $s_i:=|e_i\cap B|$. Every vertex belongs to exactly $k$ edges of a tight Hamilton cycle. Double-counting incidences between $B$ and $E(C)$ gives
\begin{equation}\label{eq:average}
\frac1n\sum_{i=1}^n s_i
=\frac{k|B|}{n}
=a+\frac12.
\end{equation}
By Lemma~\ref{lem:confinement}, all $s_i$ lie in one of the three intervals appearing there. If they all lie in the lower interval, their average is at most $a-2$; if they all lie in the upper interval, their average is at least $a+3$. Both alternatives contradict \eqref{eq:average}. Hence
\[
s_i\in\{a,a+1\}
\]
for every $i$.

Let $x$ be the number of edges of type $a+1$. Then \eqref{eq:average} gives
\[
a(n-x)+(a+1)x=\left(a+\frac12\right)n,
\]
so $x=n/2$. Edges of type $a$ are red and edges of type $a+1$ are blue, proving the claim.
\end{proof}
For completeness, we verify that the construction itself contains a tight Hamilton cycle. Thus the obstruction is genuinely a colour-balance obstruction rather than a failure of Hamiltonicity.

\begin{prop}\label{prop:hamiltonian}
If $2k$ divides $n$, then $H_{k,a}(n)$ contains a tight Hamilton cycle.
\end{prop}

\begin{proof}
Set $r:=2a+1$ and define a bi-infinite binary sequence by
\[
w_i:=\left\lfloor\frac{(i+1)r}{2k}\right\rfloor
-\left\lfloor\frac{ir}{2k}\right\rfloor,
\qquad i\in\mathbb Z.
\]
Since $1\le r<2k$, each $w_i$ belongs to $\{0,1\}$. Moreover, $w_{i+2k}=w_i$ for every $i$, and telescoping over one period shows that a period of length $2k$ contains exactly $r$ ones. Since $2k\mid n$, repeating this period $n/(2k)$ times yields a cyclic binary word of length $n$ containing exactly $\frac{rn}{2k}=|B|$
ones.

For every $i\in\mathbb Z$, telescoping over a block of length $k$ gives
\[
\sum_{j=i}^{i+k-1}w_j
=\left\lfloor\frac{(i+k)r}{2k}\right\rfloor
-\left\lfloor\frac{ir}{2k}\right\rfloor.
\]
Writing $ir/(2k)=m+\theta$ with $m\in\mathbb Z$ and $0\le\theta<1$, and using $r/2=a+1/2$, the right-hand side equals
\[
\left\lfloor \theta+a+\frac12\right\rfloor\in\{a,a+1\}.
\]
Place the vertices of $B$ in the positions occupied by ones and the vertices of $A$ in the remaining positions. Because the binary sequence is $2k$-periodic and $2k\mid n$, the same calculation applies to cyclic intervals crossing the end of the word. Hence every cyclic interval of $k$ consecutive positions has type $a$ or $a+1$, and is therefore an edge of $H_{k,a}(n)$. The resulting cyclic ordering is a tight Hamilton cycle.
\end{proof}

\section{Degree estimates}\label{sec:degree}

As usual, we interpret $\binom{x}{y}$ as zero whenever $y<0$ or $y>x$.

\begin{prop}\label{prop:degree}
Let $H:=H_{k,a}(n)$. If $v\in A$, then
\begin{equation}\label{eq:degreeA}
\begin{split}
d_H(v)=\binom{n-1}{k-1}
&-\binom{|B|}{a-1}\binom{|A|-1}{k-a}\\
&-\binom{|B|}{a+2}\binom{|A|-1}{k-a-3}.
\end{split}
\end{equation}
If $v\in B$, then
\begin{equation}\label{eq:degreeB}
\begin{split}
d_H(v)=\binom{n-1}{k-1}
&-\binom{|B|-1}{a-2}\binom{|A|}{k-a+1}\\
&-\binom{|B|-1}{a+1}\binom{|A|}{k-a-2}.
\end{split}
\end{equation}
Consequently, for fixed $k$ and $a$,
\[
\delta_1(H)=\bigl(d_{k,a}+o(1)\bigr)\binom{n}{k-1}
\]
as $n\to\infty$ through multiples of $2k$.
\end{prop}

\begin{proof}
There are $\binom{n-1}{k-1}$ $k$-sets containing a fixed vertex. Suppose first that $v\in A$. Such a $k$-set is absent from $H$ precisely when its remaining $k-1$ vertices contain either $a-1$ or $a+2$ vertices of $B$, which gives \eqref{eq:degreeA}. If $v\in B$, the remaining vertices of an omitted edge contain either $a-2$ or $a+1$ vertices of $B$, giving \eqref{eq:degreeB}.
Set $q:=\frac{2a+1}{2k},$
so that $|B|=qn$ and $|A|=(1-q)n$. For every fixed $t\in\{0,\ldots,k-1\}$, expansion in falling factorials gives
\begin{equation}\label{eq:hypergeomA}
\frac{\binom{|B|}{t}\binom{|A|-1}{k-1-t}}{\binom{n-1}{k-1}}
=\binom{k-1}{t}q^t(1-q)^{k-1-t}+O_{k,a}(n^{-1}),
\end{equation}
and similarly
\begin{equation}\label{eq:hypergeomB}
\frac{\binom{|B|-1}{t}\binom{|A|}{k-1-t}}{\binom{n-1}{k-1}}
=\binom{k-1}{t}q^t(1-q)^{k-1-t}+O_{k,a}(n^{-1}).
\end{equation}
Terms whose lower binomial parameter lies outside its natural range vanish identically. Hence, applying \eqref{eq:hypergeomA} and \eqref{eq:hypergeomB} to the remaining terms and dividing by $\binom{n-1}{k-1}$, the degree in \eqref{eq:degreeA} tends to
\[
1-\mathbb P\bigl(X_{k,a}\in\{a-1,a+2\}\bigr),
\]
while the degree in \eqref{eq:degreeB} tends to
\[
1-\mathbb P\bigl(X_{k,a}\in\{a-2,a+1\}\bigr).
\]
Their minimum is $d_{k,a}$ by \eqref{eq:dka}. Finally,
\[
\binom{n-1}{k-1}
=\bigl(1+O_k(n^{-1})\bigr)\binom{n}{k-1},
\]
which proves the stated asymptotic formula.
\end{proof}

\begin{proof}[Proof of Theorem~\ref{thm:main}]
Take $H_{k,a}(n)$ as defined in \eqref{eq:construction}. Proposition~\ref{prop:hamiltonian} gives a tight Hamilton cycle, Proposition~\ref{prop:balanced} proves that every tight Hamilton cycle has colour sum zero, and Proposition~\ref{prop:degree} gives the asserted minimum vertex degree.
\end{proof}

\subsection{The counterexample at uniformity seventeen}

\begin{proof}[Proof of Corollary~\ref{cor:k17}]
Take $k=17$ and $a=8$. Then
\[
X_{17,8}\sim\operatorname{Bin}(16,1/2).
\]
By symmetry of the binomial distribution,
\[
\mathbb P\bigl(X_{17,8}\in\{7,10\}\bigr)
=\mathbb P\bigl(X_{17,8}\in\{6,9\}\bigr),
\]
and hence
\begin{align*}
d_{17,8}
&=1-2^{-16}\left(\binom{16}{7}+\binom{16}{10}\right)\\
&=1-\frac{19448}{65536}
=\frac{5761}{8192}
\approx0.7032470703.
\end{align*}
On the other hand, \eqref{eq:dk} gives
\[
d_{17}=1-\frac{8}{17}\left(\frac{33}{34}\right)^{15}
\approx0.6992771642.
\]
More precisely,
\begin{equation}\label{eq:gap}
d_{17,8}-d_{17}
=\frac{8}{17}\left(\frac{33}{34}\right)^{15}-\frac{2431}{8192}
>\frac{39}{10000},
\end{equation}
where the last inequality follows by clearing denominators.

Fix $\alpha:=3/1000$. By Theorem~\ref{thm:main} and \eqref{eq:gap}, for every sufficiently large multiple of $34$ the graph $H_{17,8}(n)$ satisfies
\[
\delta_1\bigl(H_{17,8}(n)\bigr)
\ge(d_{17}+\alpha)\binom{n}{16}.
\]
Nevertheless, every tight Hamilton cycle in $H_{17,8}(n)$ has exactly $n/2$ edges of each colour. Therefore, for this fixed $\alpha$, the conclusion of Conjecture~\ref{conj:BCHM} fails for every $\delta>0$ and arbitrarily large $n$. This disproves the conjecture.
\end{proof}

\subsection{Large uniformity}

\begin{proof}[Proof of Corollary~\ref{cor:largek}]
If $k$ is odd, choose $a=(k-1)/2$; if $k$ is even, choose $a=k/2-1$. In the first case
\[
\frac{2a+1}{2k}=\frac12,
\]
and in the second case
\[
\frac{2a+1}{2k}=\frac12-\frac{1}{2k}.
\]
Thus, in both cases, if $X:=X_{k,a}$ then its binomial parameter is $1/2+O(k^{-1})$. In particular, this parameter stays in a fixed compact subinterval of $(0,1)$ for all sufficiently large $k$. Stirling's formula therefore gives the uniform local estimate
\[
\max_{j\in\mathbb Z}\mathbb P(X=j)
=O(k^{-1/2}).
\]
Each of the two probabilities appearing inside the maximum in \eqref{eq:dka} is a sum of two point probabilities. Hence
\[
1-d_{k,a}=O(k^{-1/2}).
\]
Since $d_{k,b}\le1$ for every $b$,
\[
0\le 1-\max_{0\le b\le k-1}d_{k,b}
\le 1-d_{k,a}
=O(k^{-1/2}),
\]
which proves the first assertion. For the second assertion, \eqref{eq:dk} gives
\[
\frac{k-1}{2k}\longrightarrow\frac12,
\qquad
\left(1-\frac{1}{2k}\right)^{k-2}\longrightarrow e^{-1/2},
\]
and therefore
\[
d_k\longrightarrow1-\frac{e^{-1/2}}2.
\]
\end{proof}

\section*{Acknowledgements}
This work was supported by the National Natural Science Foundation of China (No.~12161073), the GDUPT Talent Recruitment Project (No.~2024rcyj1007), and the Scientific Research Funds at China University of Geosciences (Wuhan) (Project No.~2026039).

\bibliographystyle{abbrv}
\bibliography{main}
\end{document}